\documentclass[a4paper,reqno]{amsart}

\usepackage[centertags]{amsmath}

\usepackage{amsmath}
\usepackage{amsfonts}
\usepackage{amssymb}
\usepackage{verbatim}
\usepackage{enumerate}

\usepackage{graphicx}

\newcommand{\calS}{{\mathcal S}}

\newcommand{\calG}{{\mathcal G}}

\newcommand{\calE}{{\mathcal E}}

\newcommand{\supp}[1]{{supp\,(#1)}}

\newcommand{\calM}{{\mathcal M}}

\newcommand{\calA}{{\mathcal A}}

\newcommand{\calC}{{\mathcal C}}

\newcommand{\N}{\mathbb N}

\newcommand{\R}{\mathbb R}

\newtheorem{theorem}{Theorem}
\newtheorem{lemma}{Lemma}
\newtheorem{rem}{Remark}
\newtheorem{prop}{Proposition}

\newtheorem{example}{Example}
\newtheorem{definition}{Definition}
\newtheorem{corollary}{Corollary}

\title[$G$-games with coalitions]{$G$-games with coalitions}

\author{Roy Cerqueti}
\address{University of Macerata, Department of Economics and Law.
Via Crescimbeni 20, I-62100, Macerata, Italy}
\email{roy.cerqueti@unimc.it}

\author{Emilio De Santis}
\address{University of Rome La Sapienza, Department of Mathematics.
Piazzale Aldo Moro, 5, I-00185, Rome, Italy}
\email{desantis@mat.uniroma1.it}

\begin{document}

\begin{abstract}
This paper models games where the strategies are nodes of a graph $G$ (we denote them as $G$-games) and in presence of coalition structures. 
The cases of one-shot and repeated games are presented. In the latter situation, coalitions are assumed to move from a strategy to another one under the constraint that they are adjacent in the graph.  
We introduce novel concepts of pure and mixed equilibria which are comparable with classical Nash and Berge equilibria.  
A Folk Theorem for $G$-games of repeated type is presented.   
Moreover, equilibria are proven to be described
through suitably defined Markov Chains, hence leading to a constrained Monte Carlo Markov Chain procedure. 
\medskip
\medskip
\newline
\emph{Keywords:} Game theory; Graphs; Repeated games;
constrained Monte Carlo Markov Chain.
\newline
\medskip
\emph{AMS MSC 2010:} 91A35, 91A40, 91A43, 60J20.
\end{abstract}

\maketitle

\section{Introduction}
 
%RISCRIVERE INTRODUZIONE E CONCLUSIONI

%\cite{animal, szabo2007, stacchetti, kreps, ely, yamamoto}

%PRESO DA KANDORI, M., R.

%The work of Nelson and Winter (which culminated in their
% book (1982)) should also be mentioned for its emphasis on the importance of
% evolutionary ideas in explaining economic change.

%\cite{Congestion,Potential}

%XXXX AGGIUNGERE REFERENCES A GAMES AND ECONOMIC BEHAVIOR XXXX

We present evolutionary game models where the strategies have a graph structure. 
The players may pass from a
strategy to an adjacent one in the graph or, alternatively, they can
hold their positions. The target of each player is the maximization
of her/his payoff, and players are allowed to form coalitions.

 \begin{comment}
This paper deals with a particular class of $n$-players games $(V,
\calS,\pi)$ where $V=\{1, \dots, n\}, \calS, \pi$ represent the set
of the players, the set of the strategies and the vector of payoffs,
respectively. In particular, we consider the special case of
strategies located at the nodes of a graph $G=(\calS,E)$ and
introduce according the notation of $G$-game $(V, \calS,\pi;G)$.
 
The only rule of our game is that the players may pass from a
strategy to an adjacent one in the graph or, alternatively, they can
hold their positions. The target of each player is the maximization
of her/his payoff, and players are allowed to form coalitions.
 
\end{comment}

Taking the strategies of the game as the nodes of a graph has an
intuitive motivation. In fact, we have in mind that the change of a
decision is usually a stepwise process over a continuum of
alternatives, where the decider modifies her/his status by selecting
a status which is close to the previous one. Think about a game where the strategies of the single players are 
the positions on a lattice, and at any time they have to decide a new position in the neighborhood of the previous one. 
Such situations commonly appear in the real world. This is the case of physical restrictions in the actions, with players constrained to move from a position (or strategy) to an adjacent one. An example might be the location strategies with time constraints: each player selects a geographical location which is achievable from her/his current position in no more than $10$ minutes walking. It is clear that the selected position will be close (adjacent, in some sense) to the previous one.

Some relevant contributions in game theory discuss the cases of interacting players, who are then viewed as nodes of a graph (see e.g. \cite{myerson} and the recent contribution \cite{szabo2007} with references therein). This setting is associated to the strategic behavior of the agents, whose connections of local types might create global behaviors. 

Less attention has been paid to the case in which strategies are linked together.
In this respect, it is important to give credit to some remarkable examples in
the literature of games with strategies exhibiting a graph structure. 
 
A graph-based structure of the strategies is
presented in the context of game colouring graphs, which is a well-known 
problem in game theory. 
In such a model, two players alternatively select and
colour a node of the graph under some adjacency constraints and with opposite targets  (for the details, see e.g. \cite{zhu}).
 
Under different conditions, we mention also the
game of cops and robber (see e.g. \cite{aigner} and the monograph \cite{bonato}). In this case, the players (cops and robbers)
move on a graph at each step to a node which is adjacent to the
previous position, with the intuitive target to escape (robber) and
to catch the robber by occupying her/his same node (cops).

The trapping games are also relevant (see \cite{trap}). Here a couple of players is considered, and they are assumed to select alternatively strategies along the arcs of a graph. The selection of each vertex twice is forbidden, so that one of the players loses the game when she/he cannot move (the player is trapped).
 
We feel that such games are close to us for one important aspect:
 the players move on strategies which are
adjacent nodes of a graph. However, we are more general than the
mentioned models, in that we abandon the binary situation
win-lose and the alternation of the players in moving, consider more than two players and admit the
presence of coalitions of players. Moreover, we present here a game theoretical modeling since we do not fix payoffs or a specific target for the players but, rather, we develop a general framework to be adapted to a plethora of special cases. 
 
%In our context, players can cooperate. In details, we assume that
%the game presents a coalition structure (see e.g.
%\cite{owen1968game, Tynyanskii1981}), i.e. the set of the players is
%partitioned. Each class of the partition is a coalition.
 
As mentioned above, in the proposed setting we consider a
partition of the set of the players in coalitions -- and in so doing,
we fix a coalition structure of the game  (see e.g.
\cite{aumanncoala, owen1968game}).

%The coalition
%structure is not here determined through an optimal partition problem -- the so-called coalition
%structure generation -- is challenging, and is grounded on the
%conceptualization and maximization of the payoff of a partition. %We
%address the interested reader to \cite{sandholm1999coalition,
%bilbao2003cooperative} XXXX CITARE DA RIVISTA XXXX
 
We introduce and study the equilibria of the game. At this aim, we move from the breakthrough work of Nash \cite{Nash51} -- which
relies to the context of non-cooperative games -- and consider the Berge
extension of the classical Nash equilibria to coalitions of players
 (see \cite{berge1957theorie}). Specifically, we extend such a concept by
including the presence of a graph structure for the strategies. In
particular, we provide a definition of equilibria on the basis of
the comparison of the payoffs of adjacent strategies, hence extending the concept of 
the standard Berge equilibria. Indeed, Berge
equilibria are subcases of our setting, in that they are included in our definition when the graph of the strategies is complete. 

Our framework is then adapted to the repeated games (see
e.g. \cite{arthur1987, aush, benoit, david2001, KMR93}).
More specifically, we consider a class of such games where the
players move on adjacent strategies in the graph from a time to the
next one. The time horizon $T$ of the repeated game can be finite or infinite.
The final payoff of each coalition is assumed to be derived by the sum
of the payoffs realized at each step of the repeated game. In the
situation of infinitely played repeated games -- sometimes called
\emph{supergames} -- scholars introduce often a discount factor
$\delta \in (0,1)$ for having the convergence of the payoffs series
(see e.g. \cite{abreu, stacchetti, fudenberg1990nash} and references
therein). The discount factor is not a mere mathematical
device: it has also the intuitive economic meaning of under
evaluating the future payoffs of the repeated games. In the present
paper, we avoid the introduction of the discount factor by removing
by definition the convergence problems of the aggregated payoffs.
Specifically, in accord to classical game theory literature, we take
the total payoff of the supergame as the liminf of the mean of the
payoffs of the one-shot games (see e.g. \cite{fudenberg1990nash}).
In so doing, we are in line with the literature dealing with
infinitely played repeated games without discount factor for the
payoffs, where the convergence problems in the formulation of the
series of the total payoff are removed by definition (see e.g. \cite{kripps}).

 %defining the payoffs The
%discounting factor is not necessary in our specific setting, as
%Definition \ref{ripetuti} highlights.

We prove that it is sufficient that the dynamics on the graph is driven by Markov chains, and this leads also to add results to the field of Monte Carlo Markov Chain (MCMC). For a wide perspective on MCMC, see \cite{mcmc1,bremo},
while a review of the endless applications of MCMC should include
relevant contributions like \cite{mcmc2, green}.

In the cointext of the Markov chains, \cite{cole} studies the equilibria of a repeated game
 in  dependence of the memory of the players. 
The authors consider that the action or strategy at 
any time depends only on the public knowledge related
to the previous $K$ stages of the game (memory $K$). 
Hence the players construct equilibria collected in  $\Gamma_K$. 
The space of equilibria $\Gamma_K$ is analyzed and compared with 
$\Gamma_\infty$, i.e. the equilibria constructed using all the past history
 (unbounded memory).

In \cite{HMP} a one-period game with unitary memory is studied ($K=1$) 
and a Folk Theorem for bargaining games is presented. 

For what concerns the memory of the repeated games, in our context we prove that the connectedeness of $G$ leads to $\Gamma_K=\Gamma_\infty$, for each $K \geq 1$. Moreover, 
we are able to derive a new version of the Folk Theorem tailored to our specific context (see also versions of Folk Theorems in  \cite{aush, benoit, chen, deb, maskinfudenberg, laclau}). In particular, we
deal with the case $T=\infty$. Differently with the standard case of
infinitely repeated games without discounting -- where Folk Theorem
states simply that any equilibrium of the one-shot game is also an
equilibrium for the repeated game (see the seminal contribution of
\cite{rubinstein}) -- we here introduce a natural constraint to let
such equilibria be consistent with the graph $G$ of the strategies.

In illustrating how the paper flows, some details on the treated topics are also provided. In particular, Section 2 contains the preliminary and notations which serve for formalizing the game models we deal with. We denote such game models as $G$-games, to point the attention on the graph of the strategies $G$, and assume that $G$ is a finite graph. Moreover, such a section is devoted to the definition of the (pure) $\calC$-equilibria for games with coalition structure $\calC$. The connection between $G$-games and standard games is also discussed.

Section 3 is a technical one. It discusses the definition of (mixed) $\calC$-equilibria in a static context. By construction, such equilibria are independent from the graph $G$ and depend only on the coalition structure $\calC$. Theorem \ref{th3} is an existence result for mixed $\calC$-equilibria for $G$-games, that is an arrangement with our notation of \cite{berge1957theorie,Nash51}.

Section 4 is divided into two subsections. 
In the first one, we focus on 
MCMC problems when Markov
chains are linked to graphs. Indeed, the presence of an adjacency constraint over the strategies leads to a constrained version of MCMC, in the sense that we will admit only nonnull transition probabilities between two adjacent states (strategies). In this context, we are able to say whenever, given a distribution $\mu$, it is possible to construct a homogenenous or a nonhomogeneous Markov chain having $\mu $ as empirical distribution (see Theorems \ref{dinamic2} and \ref{thdin2}). % Since we deal with a dynamic context which will be tailored on the constitutive ingredients of a game, 
A definition of a specific class of graphs -- the $\calC$-decomposable graphs -- is introduced in Definition \ref{decomponibile} on the basis of the strong products of graphs (see \cite{sabidussi}).  This definition will be used to introduce the repeated $G$-games.

In particular, $\calC$-decomposability leads to the possibility that each coalition of players might select a strategy only on the basis of the knowledge of the past history of the game, without the need of assuming communications among coalitions. Moreover, in this case, we do a specific construction of the MCMC that is  computationally less heavy than in the general case  (see Theorem \ref{dinamic3}).  

The second subsection represents the conclusion of the arguments developed in the previous parts of the paper. In fact, we introduce therein a dynamical setting by providing the definition of repeated $G$-games (see the general Definition \ref{ripetuti} and the more specific Definition \ref{def:repeated}). Here we are able to observe that the (pure) $\calC$-equilibria might turn out to be meaningful when a game is played $T=2$ times or when $T$ is unknown -- where \emph{unknown} should be intended in the sense that the coalitions cannot do any prediction about the end of the game.

Differently, the (mixed) $\calC$-equilibria can be used to prove a version of the Folk Theorem in the context of repeated $G$-games
with $T=\infty$
and under some conditions on the information. In particular, Definition \ref{def:CGR} formalizes
$\calC$-equilibria for the $G$-repeated games when $T=\infty$. Then,
Theorem \ref{thm:calM} states that the multidimensional Markov chain
introduced in the previous subsection can be viewed as a
$\calC$-equilibrium for the $G$-repeated games with $T=\infty$,
minimal information -- i.e., coalitions have knowledge only of their
previously selected strategies -- and in both cases of initial
strategies assigned by an external referee or selected by players
themselves. As a corollary of Theorem \ref{thm:calM} we have the
above-mentioned new version of the Folk Theorem, which guarantees
the existence of a $\calC$-equilibrium for the $G$-repeated games in
the considered framework.

\begin{comment}
\newline
For defining the concept of repeated game in our framework, we have
introduced the dynamics of the strategy selections by the coalitions
of players on the graph of the strategies. More specifically, we
have introduced a family of Markov chains with state space given by
the nodes of the graph of the strategies, and imposed that the
transition probabilities are not null only when the transition
occurs between strategies that are adjacent on the graph. In this
respect, we are dealing with a constrained Monte Carlo Markov Chain
(MCMC) procedure on graphs (XXXX INSERIRE REFERENCES XXX).
\newline
The novelties of our paper are the following: XXXX ENFATIZZARE XXXX

\end{comment}
 
 Last section provides some conclusive remarks and traces lines for future research.

\section{Definition of a $G$-game}\label{Sec2}

Consider a set of $n$ players $V=\{ 1, \ldots , n \}$.
% Players are mutually interconnected, and we collect the connections in a set $E$.
%Thus the players, along with their interconnections,
%form a graph $G =(V, E)$.
% Let us consider a collection $\mathcal{C}$ of  subsets of $V$,
%with the requirement that for each $j \in V$ there exists $C\in \mathcal{C}$ such that $j \in C$.
The $j$-th player has a set of $k_j$ (pure) strategies collected in
$\calS_j = \{   s_{j}^1,    \ldots ,      s_{j}^{k_j}  \}$. We
denote the product space of the strategies as
$$
\calS = \prod_{j \in V}\calS_j ,
$$
and an element $s \in \calS$ is a \emph{profile of strategies}.  For
$j \in V$, $\pi_j : \calS \to \R $ denotes the payoff function of
the $j$-th
 player. Hence, for $s  =(s_1, \ldots , s_n )\in \calS $, $\pi_j (s)$ is the payoff of the $j$-th
 player  when the players use the profile of strategies $s$. The \emph{vector of payoffs} is $\pi=(\pi_1, \ldots, \pi_n)$.

For convenience we introduce a notation for strategies substitution.
For a given set $C \subset V$ and $s,t \in \calS$ we
define the profile of strategies $[s, t; C] \in \calS$, by setting,
for its $j$-th component,
$$
[s, t; C] _j= \left \{ \begin{array}{cc}
                         t_j & \text{if } j \in C;  \\
                         s_j &  \text{if } j \not \in C.
                        \end{array} \right .
$$
%where $\bar{C} = V \setminus C$.

A set $C \subset V$ of players is called 
\emph{coalition}. We are interested in the coalitions which form a partition
$\calC=\{C_1, \ldots, C_r\}$ of $V$ and we will call this partitition  
a \emph{coalition structure}. We consider games
which present a \emph{coalition structure}.

The payoff of a coalition $C$ is a function
$\Pi_C:\calS \to \R$. Sometimes, it is natural to consider $\Pi_C$
as the sum of the payoffs of the single players belonging to $C$,
i.e. $\Pi_C(s) =\sum_{i \in C}\pi_i(s)$, for $s \in \calS$.

For a given coalition $C \subset V$ we write the set of the pure
strategies of $C $ as
$$
\calS_C = \prod_{\ell \in C} \calS_{\ell} . % =\{s^1_C ,s^2_C ,\ldots , s^{|\calS_C|}_C \},
$$
Let us consider a $G$-game with coalition structure $\calC =\{ C_1,
\ldots , C_r\}$. 
Let $s_{C_\ell} \in \calS_{C_\ell}$ for $\ell =1, \ldots , r$.
The profile of strategies can be written by $s =(s_{C_1} , \ldots , s_{C_r} ) \in \calS
$.  In this formalism the strategy of the $j$-th player can be obtained as the projection on $S_j $ of $s_{C_{\bar \ell}}$ where $\bar \ell \in \{ 1,2, \ldots , r  \} $ is the unique element having $j \in C_{\bar \ell }$.

%where $|\calS_C| =\prod_{\ell \in C} k_\ell$. %We enumerate
%$\calS_C$ as
%$$
%\calS_C =\{s^1_C ,s^2_C ,\ldots , s^{|\calS_C|}_C \}.
%$$
The vector of the payoffs of the coalitions  $\calC$ is denoted by  $\Pi=(\Pi_{C_1}, \ldots,\Pi_{C_r})$.

\medskip

Hereafter, we consider that the elements of  $\calS$ are nodes of a
graph $G =(\calS, E )$. Some notations for graphs are now presented.
Given two graphs $G=(\calS, E)$ and $G'=(\calS', E')$ we say that
$G'$ is a \emph{subgraph} of $G$ if $\calS'\subset \calS$ and
$E'\subset E$, and we write $G' \subset G$. Moreover, the subgraph
$G' \subset G$ is said to be an \emph{induced subgraph} of $G$ if
$s,t \in \calS'$ and $\{ s,t   \} \in E$ imply  $\{ s,t   \} \in E'
$. In this case we write $G' = G [\calS']$.

The profiles of strategies  $ s, t \in \calS$ are declared
\emph{adjacent} in $G$ if $\{s,t\}\in E$ or $s=t$.
 We define a $G$-\emph{game} with coalition $\calC $ as the quadruple $(V, \calC , \Pi, G) $, where $\Pi $ are the payoffs of the coalitions $\calC$.

A key concept of the present study is the equilibrium of the
$G$-game in presence of coalitions.

\begin{definition}\label{C-equilibrium}
Given a graph $G=(\calS,E)$ and a $G$-game $(V, \calC, \Pi    , G) $ where  $\calC$ is a partition of
$V$, we say that $\bar{s} \in \calS$ is a pure
\emph{$\mathcal{C}$-equilibrium} for the $G$-game if
\begin{equation}\label{Ceq}
  \Pi_C (\bar{s}) \geq \Pi_C ([ \bar{s},
s ; C]    ),
\end{equation}
for any $C \in \calC$  and any $s \in \calS $ such that $\{s,\bar s\} \in E$.
\end{definition}

Sometimes we refer to pure $\calC$-equilibria simply as $\calC$-equilibria.

Notice that the concept of $\calC$-equilibrium in Definition
\ref{C-equilibrium} coincides with the classical Berge equilibrium 
(see \cite{berge1957theorie})  when the considered graph $G$ is complete. Indeed,
Berge equilibrium is a Nash one for coalitions, without any
structure of the set of the strategies. Some easy remarks follow.

\begin{rem}\label{primorem}
If all the nodes of the graph $G $ are isolated, then any $s \in
\calS $ is a $\calC$-equilibrium.

If the graph $G $ is  complete and $\calC =\{ \{1\} , \ldots , \{n\}
\}$ then all the couples of elements of $\calS$ are formed by
adjacent strategies and the $G$-game becomes the standard game
$(V,\calS,\pi)$. In this case, $s \in \calS$ is a
$\calC$-equilibrium for the $G$-game $(V, \calC,\Pi, G)$ if and only
if it is a pure Nash equilibrium for $(V, \calS, \pi)$. If the graph $G
$ is complete and $\calC =\{V \}$ with 
$$
\Pi_V (s)= \sum_{i \in V} \pi_i  (s) , \hbox{ for } s \in \calS . 
$$
Then any  $\calC$-equilibrium for the $G$-game is a Pareto
optimal solution for the game $(V,\calS,\pi)$. 
The finiteness of $\calS $
 guarantees that each $G$-game admits a $\calC$-equilibrium when $\calC =\{V \}
 $ and thetotal  payoffof $V$ is the sum of the individual payoffs.

Thus, our definition of $\calC$-equilibrium for a $G$-game is not
only more general than Berge equilibrium, but it is also a
generalization of the Nash equilibrium and Pareto optimal solution for a game.
% Directly from Definition \ref{C-equilibrium}, for two $G$-games 
%$ (V, \calC , \Pi, G')$ and $ (V, \calC , \Pi, G'')$ having $G' = (\calS, E') $, $G'' = (\calS, E'' ) $
% and $ E' \subset E'' $ one has that any   $\calC$-equilibrium for $(V, \calC,\Pi,G'')$, is also a
% $\calC$-equilibrium for $(V, \calC,\Pi;G')$.
\end{rem}

For a $G$-game with a coalition structure $\calC$, we collect the $\calC$-equilibria in the set $\calE_\calC $. The set 
$\calE_\calC $  contains the set of the pure Berge equilibria but it can be empty.

\section{Mixed $\calC$-equilibria}

In this section we introduce and analyse the  mixed
equilibria for a $G$-game. The context is static, in the
sense that the game is played only one time. Therefore, it will be clear that all the mixed $\calC$-equilibrium,  here presented, will not depend on the choice of the
graph $G$. However, the relevance of the graph will be clear on the section dealing
with the repeated $G$-games. This part is analogous to Berge (for coalitions) or
Nash mixed equilibria  (see \cite{berge1957theorie, Nash51}); it is presented only for ease of
reading the paper. We notice that in \cite{berge1957theorie, Nash51} the pure equilibria
are also mixed equilibria. As we will see in our presentation this is not longer true.

A mixed strategy for a coalition $C \in \calC$ is a
distribution on $  \calS_C    $, and we denote it by $\Lambda_C =(
\lambda_{C } (s) : s \in  \calS_C  ) $.
In presence of mixed strategies,   the payoff of a coalition $C $ will be a random variable.
 Reasonably, we will also consider
that the single coalitions act independently one each other because they are not comunicating. Thus,
the choice of the individual mixed strategies for all the
coalitions $C_1, \ldots, C_r$ fixes also a product distribution on
$\calS$ for the game. Hence, the expected payoff for coalition $C \in \calC $ is
\begin{equation}\label{expected}
    \mathbb{E}_{\Lambda_{C_1} \times \ldots \times
    \Lambda_{C_r}} (\Pi_C ) = \sum_{s_{C_1 } \in \calS_{C_1}} \cdots
\sum_{s_{C_r } \in \calS_{C_r}}    \Pi_C
(s_{C_1} , \ldots ,  s_{C_r} ) \left
[\prod_{\ell=1}^r \lambda_{C_\ell } (s_{C_\ell}  )\right ],
\end{equation}
where $\mathbb{E}_{\Lambda_{C_1}\times \ldots \times \Lambda_{C_r}}$
is the expected value with respect to the product distribution
$\Lambda_{C_1}\times \ldots \times \Lambda_{C_r}$ and $\Lambda_{C_\ell}$, for $\ell =1, \ldots , r$,  
is the distribution selected by the coalition $C_\ell$.

\begin{definition}\label{mixed}
Let us consider a $G$-game $(V,\calC, \Pi,G) $ with coalition
structure $\calC =\{ C_1, \ldots , C_r\}$. A mixed
$\calC$-equilibrium for the $G$-game is a product distribution $\bar
\Lambda=\prod_{h=1}^r\bar \Lambda_{C_h} $, where
\begin{equation}\label{mixdefstr}
\mathbb{E}_{\bar \Lambda_{C_1} \times \ldots \times \bar
\Lambda_{C_r}} (\Pi_{C_\ell}) \geq     \mathbb{E}_{\bar
\Lambda_{C_1}\times \ldots \times \bar \Lambda_{C_{\ell-1}}\times
\Lambda_{C_\ell}\times \bar \Lambda_{C_{\ell+1}}\times \ldots\times
\bar \Lambda_{C_{r}}}(\Pi_{C_\ell}),
\end{equation}
for any $\ell \in \{1, \ldots, r\}$ and for any distribution $ \Lambda_{C_\ell}$
  on the space $\calS_{C_\ell}$.
%\newline

We collect
 all the mixed $\calC$-equilibria in the set $\calM_\calC $. 
\end{definition}

%If we imagine that each coalition $C \in \calC$ is a single player with set of pure strategies $\calS_{C}$ 
%of a new game, then the
%$G$-game becomes a game in the definition of Nash. 
%In
%our situation the mixed strategy of a coalition
%coincides with the mixed strategy of a player in \cite{Nash51}.
%Also \eqref{expected} coincides with the unique linear extension of
%the payoffs given on the pure strategies (as in \cite{Nash51}).

As already announced above, Definition \ref{mixed} does not depend
on the presence of the graph $G$, and can be provided for a generic
game whose strategies are not nodes of a graph. However, the reference to $G$-games will be
useful in the next Section.
\begin{theorem}\label{th3}
For a given $G$-game $(V,\calC, \Pi, G) $ there exists a mixed $\calC$-equilibrium.
\end{theorem}
\begin{proof}
Let us see the coalition $C \in \calC $ as a single player having space 
of strategies given by $\calS_C$.   The payoff of the player identified with the coalition $C$ is $\Pi_C$. Therefore we are dealing with a standard game as defined by Nash in \cite{Nash51}. 
By applying Theorem 1 of \cite{Nash51} one obtains the existence of a
product distribution  $\bar \Lambda = \prod_{i=1}^r \bar
\Lambda_{C_i}$ which satisfies \eqref{mixdefstr}.
\end{proof}

For a given $G$-game $(V, \calC, \Pi , G)$,  let us define 
$$
\hat \calM_\calC =\{(s_{C_1}, \ldots, s_{C_r}) \in \calS : \delta_{s_{C_1}} \times \ldots \times \delta_{s_{C_r}} \in \calM_\calC  \} . 
$$
We notice that $\calE_\calC \cap \hat \calM_\calC$ is the set of the pure Berge equilibria. Furthermore, in general, we have that $\calE_\calC \not \subset \hat  \calM_\calC$ and $ \hat \calM_\calC \not \subset \calE_\calC$. 

%that include also the pure equilibria but not necessarily the $\calC$-equilibria)

\section{Monte Carlo Markov Chain on graphs and repeated $G$-games}

This section deals with a constrained MCMC procedure in presence of
graphs and its application to the \emph{repeated
$G$-games}, which will be defined below. The MCMC part can be treated separately and it has a relevance and an interest in itself.

\subsection{MCMC on graphs}
We deal with a MCMC problem. In particular, we construct some Markov chains which are linked
to the graph $G$.

\begin{definition}\label{dinamica}
We say that a stochastic process  $X =( X(t) : t \in \mathbb{N})$ on
$\calS$ is \emph{consistent} with the graph $G=(\calS, E) $ if, for each $t\in \mathbb{N}$,      $ X(t)$
and $ X(t+1) $ are adjacent in $G$ with probability one. 
\end{definition}

We notice that if a process $X=( X(t) : t \in \mathbb{N})$ is
consistent with a graph $G $ then it is also consistent with any
graph $G' \supset G$.

%As we will see, the processes consistent with a graph $G$ will turn
%out to be useful for defining the equilibria of a repeated $G$-game.

Given  a finite graph $G=(\mathcal{S},E)$ and a  distribution
$\mu=(\mu(s):s \in \calS)$,  we will provid an answer to the following question:

\begin{itemize}
\item[\textbf{Q}] \emph{Is it possible to construct a (not
necessarily homogeneous) Markov chain $X =( X(t) : t \in
\mathbb{N})$ which is consistent with $G$ and such that its
empirical distribution converges almost surely to $\mu$ as $t$ goes to infinity?}
\end{itemize}

Thus, we want to construct  a Markov chain $X =( X(t) : t \in
\mathbb{N})$  with the following properties:
 $X$ is consistent with the graph $G$ and 
 \begin{equation}\label{limite1bis} 
\lim_{t \to \infty}  \frac{1}{t} \sum_{m=0}^{t-1} \mathbf{1}_{\{ X
(m)=s \}} = \mu(s)  , \qquad s \in \calS \qquad a.s..
\end{equation}

We provide an answer to question \textbf{Q}  in all 
possible situations and we show that when $G$ is connect it is possible to construct such a Markov chain. 
%Under some not too restrictive conditions, we will provide a %positive answer to question \textbf{Q}. 
Moreover, in constructing
such a Markov chain, we are in the framework of the MCMC theory, even if
here the Markov chain is constrained to have transitions only between adjacent states of  $G$.

All the possible situations, along with the related answers to question \textbf{Q}, can be distinguished in four cases:
\begin{itemize}
\item[$(i)$] If the distribution $\mu$ is concentrated on a unique $\bar s \in \calS$, i.e. $\mu=\delta_{\bar
s}$, then one can construct the constant Markov chain $X=(X(t):t \in
\mathbb{N})$ such that $X(t)=\bar s$, for each $t$. By Definition
\ref{dinamica} and the concept of adjacent states, one has that $X$
is consistent with $G$ and \eqref{limite1bis} is trivially satisfied.
\item[$(ii)$] If $G[\supp{\mu}]$ is not connected but $\supp{\mu}$ is contained in a connected
component of $G$, then one can construct a nonhomogeneous Markov
chain which is consistent with $G$ and fulfilling condition \eqref{limite1bis}
(see Theorem \ref{dinamic2} below).
\item[$(iii)$] If $G[\supp{\mu}]$ is not connected and $\supp{\mu}$ is not contained in a unique connected component of $G$, then it does not exist a stochastic
process which is consistent with $G$ and fulfilling
\eqref{limite1bis} (see Theorem \ref{thdin2} part b. below). 
\item[$(iv)$] If $G[\supp{\mu}]$ is connected, then one
can construct a homogeneous Markov chain consistent with $G$ which
satisfies \eqref{limite1bis} (see Theorem \ref{thdin2} part a. below).
\end{itemize}

We now deal with item  $(ii)$. 

Notice that, in this case, there exists a connected
component of $\calS$, say $\hat \calS$, such that $\supp{\mu}
\subset \hat \calS$ and $\supp{\mu}
\neq \hat \calS $. Without loss of generality and to avoid
the introduction of further notation, we assume that $G$ is
connected and we identify  $\hat \calS $ 
with $ \calS$.

For a given distribution $\mu=(\mu(s):s \in \calS)$, let us define, in case $(ii)$, the non-empty set 
$$
\calA_k=\left \{s \in \calS: \mu(s) <\frac{1}{k} \right \},\qquad k
\geq 1
$$
and let the distribution $\eta_k=(\eta_k(s):s \in \calS)$ be
$$
\eta_k (s) = \frac{1}{|\calA_k|} \mathbf{1}_{\{ s \in \calA_k \}},
\qquad s \in \calS,
$$
i.e. $\eta_k$ is the uniform distribution on $\calA_k$. We also
define the distribution $\mu_k=(\mu_k(s):s \in \calS)$ as
\begin{equation}
\label{mukka} \mu_k = \frac{1}{k} \eta_k + \frac{k-1}{k} \mu .
\end{equation}
Notice that
$$
|| \mu_k  - \mu ||_{TV} =  \frac{1}{k}     || \eta_k  - \mu ||_{TV}
\leq \frac{1}{k} ,
$$
where $  || \cdot ||_{TV} $ is the total variation norm (see e.g. \cite{lindvall}).

Let $N$ denote the cardinality of $\calS$. Since $\supp{\mu}$ is not connected in $G$, then it contains at least two points. Since $\supp{\mu} \subset \calS$ and $\calS$ is connected, then $N \geq 3$. 
  By construction, for $k $  large enough and since $ N \geq 3$, one
has that
\begin{equation}\label{31ott}
\mu_k (s) \geq \frac{1}{(N-1) k}, \qquad s \in \calS .
\end{equation}
Let us label the elements of $\calS=\{ s^1, \ldots , s^{N} \}$ such that 
$$
\mu( s^1) \geq \mu( s^2)\geq \cdots \geq \mu( s^N) .
$$
Let us take an integer $k $ such that 
$$
k >  \frac{1}{  \min \{ \mu
(s) >0 : s \in \calS\}} .
$$
According to definition \eqref{mukka}, with the previous selection of $k $,  one has 
% the order of $\mu $ on $\calS$ is maintained  also for $\mu_k$, i.e.
\begin{equation}\label{muk}
    \mu_k( s^1) \geq \mu_k( s^2)\geq \cdots \geq \mu_k( s^N) >0 .
\end{equation}

We now construct the transition matrix $P^{(\mu_k,G)}= (p_{l, r}: l,r = 1,
\ldots , N)$ related to the distribution $\mu_k$ and to the graph
$G=(\calS,E)$.  For each $l,m=1,\dots,N$, 
\begin{equation}\label{ordinati}
    p_{l,m } = \left\{
\begin{array}{ll}
    p, & \hbox{if $l<m$ and $\{s^l,s^m\} \in E$;} \\
    \frac{\mu_k(s^m)}{\mu_k(s^l)}p, & \hbox{if $l>m$ and $\{s^l,s^m\} \in E$;} \\
  p_{l}, & \hbox{if $l=m$;} \\
    0, & \hbox{otherwise,} \\
\end{array}
\right.
\end{equation}
where
$$
 p_{l} =1- p[\sum_{m': m'> l } \mathbf{1}_{\{ \{s^l,s^{m'}\} \in E \}}
 + \sum_{m':m'<l}\frac{\mu_k({s^{m'}})}{\mu_k(s^l)}
\mathbf{1}_{\{ \{s^l,s^{m'}\} \in E \}} ]
$$
and
\begin{equation}\label{p}
p= \min_{l=1, \ldots , N}\frac{1}{2\left(\sum_{m':m'>l}
\mathbf{1}_{\{ \{s^l,s^{m'}\} \in E \}}
+\sum_{m':m'<l}\frac{\mu_k(s^{m'})}{\mu_k(s^l)} \mathbf{1}_{\{
\{s^l,s^{m'}\} \in E \}} \right)}   .
\end{equation}
Notice that by definition $p\leq \frac{1}{2}$. In fact, since $G$ is connected, there exists
at least an edge $\{ s^1,s^{m}\} \in E$, with $m > 1$; thus the
denominator of \eqref{p} is at least equal to $  2$, when $l =1$.
The transition matrix
$P^{(\mu_k,G)}$ is well defined, since $G$ is connected.

 Formula \eqref{ordinati} assures that the couple
$(\mu_k, P^{(\mu_k,G)})$ is reversible.  Moreover, $P^{(\mu_k,G)} $
is irreducible, since $G$ is connected, thus $\mu_k$ is the unique
invariant distribution of $P^{(\mu_k,G)}$. One can also see that
$P^{(\mu_k,G)} $ is aperiodic since, by construction, $p_{l} \geq
\frac{1}{2}$ for $l=1, \ldots , N$.

We introduce the ergodic coefficient of Dobrushin  (see \cite{dobru} and  \cite{bremo} p. 235),
which is defined as 
\begin{equation}\label{dobru}
    \delta(P)=1-\inf_{i,j =1, \dots, N} \sum_{h =1}^Np_{i,h}\wedge p_{j,h}
\end{equation}
where $P=(p_{i,j}:i,j =1, \dots, N)$ is a stochastic matrix. 
%The usefulness of the Dobrushin coefficient will be clear soon.
%
%A technical lemma is now needed.
\begin{lemma}
\label{lem:D} Given the transition matrix   $P^{(\mu_k,G)}$ on $\calS$ constructed above, with $N = |\calS| \geq 3$, 
the Dobrushin's ergodic coefficient can be bounded from above as follows
$$
\delta((P^{(\mu_k,G)})^{N-1}) \leq 1- \left (    \frac{c_N}{k}
\right )^{N-1}
$$
where  $c_N =\frac{1}{2(N-1)^2} $ and $k $ is large enough.
\end{lemma}
\begin{proof}
For $k $ large enough, 
condition  $N \geq 3 $, and inequalities  \eqref{31ott} and \eqref{muk} provide
\begin{equation}\label{musumu}
    1 \leq \frac{\mu_k(s^{m})}{\mu_k(s^l)} \leq k(N-1), \qquad \text{
for } l>m.
\end{equation}
Thus, by %\eqref{p} and 
\eqref{musumu}  one obtains
$p \geq \frac{c_N}{ k}$, for $k $ large enough. 
 Then one has that, if $p_{l,m}\not=0$, 
\begin{equation}
\label{p2} p_{l,m} \geq  \frac{c_N}{k}  .
\end{equation}
For $k $ large enough, 
 since the graph $G$ is connected and $p_{l} \geq \frac{1}{2}$  for each $l = 1, \ldots , N$, 
then \eqref{p2} gives that
$$
p_{l,m}^{(N-1)} \geq     \left (    \frac{c_N}{k}    \right )^{N-1}
, \qquad l,m =1, \dots, N,
$$
where $p_{l,m}^{(N-1)}$ is the transition probability from $s^l$ to
$s^m $ in $(N-1)$ steps.
\newline
Then, by definition of the ergodic coefficient of Dobrushin in \eqref{dobru}, one has the thesis.
\end{proof}

For a given distribution over $\calS$, namely $\lambda
=(\lambda(s): s\in \calS)$,  we construct a non-homogeneous Markov
chain $X=(X(t):t \in \mathbb{N})$ with $\lambda$ as initial
distribution.  The transition matrix of the Markov chain $X$ at time $t \in \mathbb{N}$
will be  denoted by $P(t)=(p_{i,j} (t):i,j=1, \dots,N)$. %As we will see, the Markov chain
%$X$ will be consistent with $G$ by construction of the transition
%matrix, hence leading to a first step for a positive answer to question \textbf{Q}. 

Let us consider
an  increasing sequence of times $(t_\ell:\ell \in
\mathbb{N})$, and pose 
\begin{equation}\label{Penne}
    P(t) = \sum_{k =1}^{\infty} P^{(\mu_k,G)} \mathbf{1}_{\{ t \in [t_k
, t_{k+1} ) \}} .
\end{equation}

%
% By the ergodic theorem for Markov chains one has that
%\begin{equation}\label{limo}
%    \lim_{t\to \infty} \frac{1}{t} \sum_{n=0}^{t-1} \mathbf{1}_{ \{ X(n)
%=s \}} = \mu_k(s), \qquad  s \in \calS,  \qquad a.s..
%\end{equation}
%For the given $ \mu $ we can find a probability  distribution $\calD
%=(d(s) : s \in \calS)$ with $d(s) >0  $ for each $s \in \calS$ and
%\begin{equation}\label{dmenomu}
%    |d(s) -\mu(s)| \leq \varepsilon , \qquad s\in \calS.
%\end{equation}
%By \eqref{limo} and \eqref{dmenomu} one obtains \eqref{limite1}.

\begin{theorem}\label{dinamic2}
Given a connected graph $G=(\calS,E) $ and a distribution
$\mu=(\mu(s): s\in \calS)$, any Markov chain $X =( X(t) : t \in
\mathbb{N})$ constructed above with sequence $(t_\ell:\ell \in \N)$ and  $
t_\ell=\ell^{5N}$, for $\ell \in \N$, satisfies \eqref{limite1bis}.
\end{theorem}
\begin{proof}
To prove the result, we first check that %\eqref{limite1bis} when $t $ is
%replaced with $t_\ell $ of \eqref{tell}, i.e.
\begin{equation}\label{limite222}
\lim_{\ell \to \infty}  \frac{1}{t_{\ell}} \sum_{m=0}^{t_{\ell}-1}
\mathbf{1}_{\{ X (m)=s \}} =\mu(s)  , \qquad s \in \calS \qquad
a.s..
\end{equation}
By definition of $(t_\ell:\ell \in \N)$ in the hypotheses, one
has
$$
\lim_{\ell \to \infty} \frac{t_{\ell+1}-t_\ell}{t_{\ell}} =0,
$$
and then \eqref{limite222} is equivalent to \eqref{limite1bis}.

For $\varepsilon >0$ and $s \in \calS$ let us define the sequence of events $(
B_\ell (\varepsilon, s) : \ell \in \N)$ as
\begin{equation}\label{BBX}
  B_{\ell}  (\varepsilon, s) = \left \{ \Big |\mu (s) - \frac{1}{t_{\ell+1} - t_{\ell}}
  \sum_{m= t_\ell}^{t_{\ell+1} -1} \mathbf{1}_{ \{ X(m) =s \}}\Big  | < \varepsilon \right
  \} .
\end{equation}
To obtain \eqref{limite222} it is enough that, for each $
\varepsilon
>0$ and $s \in \calS$ one has
$$
\mathbb{P} \left  ( \liminf_{\ell \to \infty }   B_{\ell}
(\varepsilon, s) 
\right ) =1 .
$$
Now, take the auxiliary independent random variables $ (Y(t):t \in
\N)$ with values on $\calS$ such that $Y(i)$ has distribution
$\mu_k$ if $i \in [t_k , t_{k+1})$ (see \eqref{mukka} for the
definition of $\mu_k$).

Notice that for each initial distribution $\vartheta$ on $\calS$, 
 Lemma \ref{lem:D} and 
Dobrushin's Theorem (see  \cite{bremo}) give that
\begin{equation}\label{Dobr2}
    || \vartheta P(t_\ell)^{\ell^{2N}}-\mu_\ell ||_{TV}
    \leq \delta(P(t_\ell)^{N-1})^{\lfloor \frac{\ell^{2N}}{N-1}
    \rfloor} \leq \left( 1-\left(\frac{c_N}{\ell}\right)^{N-1}   \right )^{ \left \lfloor \frac{\ell^{2N}}{N-1}
 \right   \rfloor} \leq \exp \left (-c_N^{N-1} \left
\lfloor\frac{\ell^{N+1}}{N-1} \right \rfloor \right ), 
\end{equation}
for $\ell $ large enough. 

Let $\hat c_N=c_N^{N-1}$. Given $i \geq 0$ and $k \geq 1 $, by
the maximal coupling (see \cite{lindvall}) and inequality \eqref{Dobr2} one can couple $X(
t_{\ell}+ k \ell^{2N} +i )$ with $ Y( t_{\ell}+ k \ell^{2N} +i ) $
so that
\begin{equation}\label{X=Y}
     \mathbb{P}
( X( t_{\ell}+ k \ell^{2N} +i ) \neq  Y( t_{\ell}+ k \ell^{2N} +i )
) \leq \exp \left (-\hat{c}_N \left \lfloor\frac{\ell^{N+1}}{N-1}
\right \rfloor \right ),
\end{equation}
for $\ell $ large enough.

Let us define the sequence of events $( A_{\ell,i}: \ell \in \N , i
\in [0, \ell^{2N} ) )$ by

\begin{equation}\label{Aelle}
A_{\ell,i}=\left\{ X(t_{\ell}+ a \ell^{2N} +i) =Y(t_{\ell}+ a
\ell^{2N} +i):  a\geq 1\, \text{ and }  t_{\ell}+ k \ell^{2N} +i
\leq t_{\ell+1}-1 \right\},
\end{equation}
for each $\ell \in \N$ and $i\in [0, \ell^{2N} )$.

For $\ell $ large enough and by subadditivity, one has
$$
\mathbb{P} ( A_{\ell,i})  \geq 1 - (\ell +1)^{5N} \exp \left
(-\hat{c}_N \left \lfloor\frac{\ell^{N+1}}{N-1} \right \rfloor
\right ).
$$
We also set $\hat A_\ell =\bigcap_{i =0}^{\ell^{2N}-1 } A_{\ell,i}
$. Then, for $\ell $ large enough,
\begin{equation}\label{Acappuccio2}
    \mathbb{P} ( \{X(t) = Y(t) : t \in  [ t_\ell + \ell^{2N}, t_{\ell+1 }
) \}) = \mathbb{P} ( \hat A_\ell  )
 \geq 1 - (\ell +1)^{7N} \exp \left (-\hat{c}_N \left
\lfloor\frac{\ell^{N+1}}{N-1} \right \rfloor \right ).
\end{equation}
By \eqref{Acappuccio2} and the first Borel-Cantelli lemma, it follows
that $\mathbb{P}(\liminf_{\ell \to \infty} \hat A_\ell) =1$.

Now, for $\varepsilon >0 $ and $s \in \calS$, let us define the
sequence of events $( \hat B_\ell  (\varepsilon ,s): \ell \in \N)$
as
\begin{equation}\label{BBY}
\hat   B_{\ell}     (\varepsilon,s ) = \left \{ \Big | \mu(s) -
\frac{1}{t_{\ell+1} - t_{\ell}}
  \sum_{m= t_\ell}^{t_{\ell+1} -1} \mathbf{1}_{ \{ Y(m) =s \}}\Big  | < \frac{\varepsilon}{2}\right
  \} .
\end{equation}
A straightforward calculation gives that
$$
\liminf_{\ell \to \infty } (\hat B_\ell    (\varepsilon ,s)    \cap
\hat A_\ell  ) \subset \liminf_{\ell \to \infty }  B_\ell
(\varepsilon , s) .
$$
Therefore to end the proof it is enough to show that $ \mathbb{P} (
\liminf_{\ell \to \infty } \hat B_\ell (\varepsilon, s ) ) =1$. Such a result is a consequence of the large deviation bounds for i.i.d.
Bernoulli random variables and the first Borel-Cantelli lemma. This
concludes the proof.
\end{proof}

\begin{rem} \label{remgenerale}
The definition of $(t_\ell: \ell \in \N)$ provided in Theorem \ref{dinamic2}
represents only one of the possible choices. In this respect, it is
interesting to note that the proof of Theorem \ref{dinamic2} can be
adapted to other sequences $(t_\ell: \ell \in \N)$. For example, one
can take $t_{\ell+1}-t_\ell \geq c \ell^{5N-1}$, with $c >0$. In
this case, for any $\ell \in \N$, there exists $I_\ell \in \N$ and an increasing sequence 
$$
t_\ell^{(0)} ,  t_\ell^{(1)} , \ldots , t_\ell^{(I_\ell)}  
$$
such that $t_\ell=t_\ell^{(0)} $,  $t_\ell^{(I_\ell)}=t_{\ell+1}$
and the following property holds 
$$
\lim_{\ell \to \infty} \sup_{i \in \{0,1,\dots, I_\ell-1\}}
\frac{t_\ell^{(i+1)}-t_\ell^{(i)}}{t_\ell^{(i)}} =0;\qquad
\lim_{\ell \to \infty}
\frac{t_\ell^{(0)}-t_{\ell-1}^{(I_\ell-1)}}{t_{\ell-1}^{(I_\ell-1)}}=0.
$$
By reproducing the arguments of the proof of Theorem \ref{dinamic2}  for the new sequences $
t_\ell^{(0)} ,  t_\ell^{(1)} , \ldots , t_\ell^{(I_\ell)}  
$, one obtains
that a new Markov chain defined with this new sequence of times satisfies \eqref{limite1bis}.
\end{rem}
Next example shows that the convergence of the distribution $\mu_k$
to the distribution $\mu$ should not be taken  \emph{too fast} and
$t_{\ell+1}-t_\ell$ should be not taken \emph{too small} in order to have
\eqref{limite1bis}. %In fact, we are dealing with the same kind of problems as for the annealing where one desires a fast convergence but it goes in contrast with the possibility to find the right solution.  
\begin{example}\label{controes}
Let us consider a graph $G=(\calS,E)$ with
$\calS=\{s^1,s^2,s^3,s^4\}$ and
$E=\{\{s^1,s^3\},\{s^3,s^4\},\{s^2,s^4\}\}$.

Let us take the distribution $\mu=(\mu(s):s \in \calS)$ having 
$\mu(s^1)=\mu(s^2)=\frac{1}{2}$, and  define $t_\ell=\ell$, for each $\ell \in \N$, and 
the sequence of distributions 
$(\hat \mu_\ell  : \ell \in \N)$ where 
$\hat \mu_\ell=\mu_{2^\ell} $. 
We take a non-homogeneous Markov chain $X=(X(t):t \in
\N)$ with transition matrix $P(\ell) =(p_{m,n} (\ell) : m,n =1, 2,3,4 )$,  at time $\ell $, given by
$$
P(\ell)=P^{(\hat \mu_\ell, G)}, \qquad \ell \in \N.
$$
In particular, $||\hat \mu_\ell-\mu||_{TV} \leq \frac{1}{2^\ell}$.

A straightforward computation gives that at time $\ell$
\begin{equation}\label{pbis}
p=\frac{1}{2^{\ell+1}},
\end{equation}
accordingly to the definition of $p$ given in \eqref{p}. Thus, \eqref{pbis} gives that
$p_{1,1}(\ell)=1-\frac{1}{2^{\ell+1}}$ at time $\ell$. Therefore,
the Borel-Cantelli's Lemma guarantees that
$$
|\{\ell \in \N : X(\ell)=s^1, X(\ell+1) \not= s^1\}|<\infty \qquad
a.s.,
$$
and therefore
$$
\mathbb{P} (\bigcap_{s \in \calS }\{\lim_{t \to \infty}  \frac{1}{t} \sum_{m=0}^{t-1} \mathbf{1}_{\{ X
(m)=s \}} = \mu(s) \}) =0.
$$
\end{example}

%The Markov chain with transition matrix 
%$P^{(\mu_k,G)}$ can be seen as a perturbation of an \emph{ideal} Markov chain whose stationary distribution is $\mu$. 
%Therefore, we are close to the important theory of Freidlin and Wentzell (see \cite{fre}).

Notice that Example \ref{controes} gives a natural comparison
between our setting and the simulated annealing (see \cite{k1983}). In both cases the
hope is that the rate of convergence is fast but, if one
tries to have an \emph{excessively high} rate of convergence, it leads to local minima (case of simulated
annealing) or not convergence to the distribution $\mu$ (case of our framework). In this case, the response to question 
\textbf{Q} might be wrong, even if the Markov chain is consistent with the graph $G$.

Next result provides an answer to \textbf{Q} for items $(iii)$ and $(iv)$.

\begin{theorem}\label{thdin2}
The following two sentences hold true:
\begin{itemize}
    \item[a.] if $G[\supp{\mu}]$ is connected, then each  homogeneous Markov chain
$X=(X(t) : t \in \N)$ with state space $\supp{ \mu }$ having
transition matrix equal to $P^{(\mu , G[\supp{\mu}])}$ defined in
\eqref{ordinati} % and initial distribution $\lambda =(\lambda (s) : s
%\in \calS)$ with $ \supp{\lambda }\subset \supp{\mu}   $
satisfies \eqref{limite1bis}. Furthermore, $X $ is consistent with
$G$;
    \item[b.] if $G[\supp{\mu}]$ is not connected, then each homogeneous Markov
chain consistent with $G$ does not  satisfies \eqref{limite1bis}.
\end{itemize}
\end{theorem}
\begin{proof} We prove a.. Since $G[\supp{\mu}]$ is connected, then the transition matrix $P^{(\mu ,
G[\supp{\mu}])}$ is well defined. Moreover, $\mu $ is the unique
invariant distribution of $ P^{(\mu , G[\supp{\mu}])} $ because $ P^{(\mu , G[\supp{\mu}])} $  is
irreducible. Now, by applying the ergodic theorem, one has
\eqref{limite1bis}. The consistence of $X$ with $G$ follows from the
fact that, for $l \not= m$, $p_{l,m}
>0 $ implies $\{s^l,s^m\} \in E $.

We  prove b. by contradiction. Assume that \eqref{limite1bis}
holds true for a Markov chain $(X(t) :t \in \N)$. Then for each $s \in \supp {\mu }$ one has
\begin{equation}\label{infinitamente}
    \mathbb{P} ( \{ X(t ) =s, \,\,\, i.o.\} ) =1 .
\end{equation}
Let us consider $ s' , s'' \in \supp{\mu} $ which belong to two
different connected components of $G[\supp{\mu} ] $. By
\eqref{infinitamente}, it  follows that $\mathbb{P} (T < \infty) =1$
where 
$$
T = \inf \{ t\in \N : X (t) \in \{ s', s''\}\}   .
$$
 Without loss
of generality one can assume that $\mathbb{P} ( X (T) = s') >0$.
Then, by the consistence of $X $ with the graph $G$, one has that
$$
\mathbb{P} (  \{t \in \N: X(t ) =s'' \} = \emptyset | X (T) = s' )
=1.
$$
 Therefore
 $$
 \mathbb{P} ( \{ X(t ) =s'', \,\,\, i.o.\} ) <1 ,
 $$
and this contradicts \eqref{infinitamente}.
\end{proof}

\begin{rem}\label{osserva} 
We notice that, by Theorem \ref{thdin2} a., it is possible, for any
$\varepsilon >0 $, to select a homogeneous Markov chain $X=(X(t) : t
\in \N)$ having transition matrix equal to  $P^{(\mu_k , G)}$ (see
\eqref{mukka} and \eqref{ordinati}), with $k \geq \lceil
\frac{1}{\varepsilon}\rceil$,
 satisfying
 \begin{equation}\label{ergo}
\lim_{t \to \infty}  \left | \frac{1}{t} \sum_{m=0}^{t-1}
\mathbf{1}_{\{ X (m)=s \}} - \mu(s)  \right | \leq \varepsilon  ,
\qquad s \in \calS \qquad a.s..
\end{equation}
  Furthermore, $X $ is
consistent with $G$.

%By the arguments of the ergodic theory 
Some consequences of Theorems \ref{dinamic2} and \ref{thdin2} arise.
Let us consider a function $f : \calS \to \mathbb{R}$ %and denote $\hat f =
% \max_{s \in \calS} |f(s)|< \infty$.

Under condition of Theorem \ref{dinamic2} or of Theorem \ref{thdin2}
a. one obtains
\begin{equation}\label{ergo1}
\lim_{t \to \infty }  \frac{1}{t} \sum_{m=0}^{t-1} f( X (m)) =
\mathbb{E}_\mu (f), \qquad a.s.,
\end{equation}
where $\mathbb{E}_\mu$ is the expected value with respect
to the distribution $\mu$, i.e.
$$
\mathbb{E}_\mu (f)=\sum_{s \in \calS}f(s)\mu(s).
$$
If \eqref{ergo} holds true, then
\begin{equation}\label{ergo2}
    \lim_{t \to \infty } \left | \frac{1}{t} \sum_{m=0}^{t-1} f( X (m))
- \mathbb{E}_\mu (f) \right | \leq  \varepsilon     \max_{s \in \calS} |f(s)|       , \qquad a.s..
\end{equation}
\end{rem}

We now need  the definition of product of graphs. The usefulness of such a definition will be clear in the next section on the repeated games. Thus, in the light of the subsequent definitions and results, we use the same notation employed in the formalization of the games. 
\begin{definition}     \label{decomponibile}
Given a finite set $V$, a 
partition $\calC =\{
C_1, \ldots , C_r\}$ of $V$ and a connected finite 
graph $G= (\calS , E) $ where $\calS = \prod_{j \in V} \calS_j $, we say that $G$ is
$\calC$-\emph{decomposable} if $G=G_1 \otimes G_2 \otimes \ldots
\otimes G_r$, where $G_1= (\calS_{C_1},E_1), \dots, G_r=
(\calS_{C_r},E_r)$ and $\otimes$ is the strong
product for graphs introduced by \cite{sabidussi}, i.e.  given $(s_{C_1} , \ldots , s_{C_r} )  ,  (\bar s_{C_1} , \ldots , \bar s_{C_r} )  
\in \calS$ they are adjacent with respect to $G$ if and only if for any $\ell = 1, \ldots , r$ the vertices $s_{C_\ell }, \bar s_{C_\ell } \in \calS_{C_\ell}$ are   adjacent with respect to $G_\ell$.
 We say that $\mathcal{G}=(G_1, \dots, G_r)$ is the
$\calC$-\emph{decomposition} of $G$.
\end{definition}
We point out that, for a given coalition structure  $\calC$, if the graph $G $ has a $\calC$-decomposition 
 $\mathcal{G}'=(G'_1, \dots, G'_r)$ then it is unique.

%We  point out that, if $\calC$ is given and $\mathcal{G}'=(G'_1, \dots, G'_r)$ is a
%$\calC$-decomposition of $G$, with $G'_1=
%(\calS_{C_1},E'_1), \dots, G'_r= (\calS_{C_r},E'_r)$, then $E_\ell=E'_\ell$ for each
%$\ell =1, \ldots, r$. This is a form of uniqueness of $\mathcal{G}$. 
We also notice that a complete graph is trivially
$\calC$-decomposable, for each partition $\calC =\{C_1, \ldots , C_r\}$ of $V$. In this
case, the graphs $G_1, \ldots, G_r$ of $\mathcal{G}$ are complete.

We consider $\calC=\{C_1, \dots, C_r\}$ a partition of $V$, a
$\calC$-decomposable graph $G=(\calS, E)$ with $\calC$-decomposition
given by $\calG=(G_1, \dots, G_r)$, where $G_h= (\calS_{C_h},E_h)$,
for each $h=1, \dots, r$. Let us take a product distribution $\mu=\prod_{h=1}^r \mu_{C_h}$, where
$\mu_{C_h}$ is a distribution on the space $\calS_{C_h}$.

In order to proceed, we construct $r$ independent Markov chains
$X_{C_1}=(X_{C_1}(t):t \in \N), \dots, X_{C_r}=(X_{C_r}(t):t \in
\N)$ such that the $h$-th Markov chain $X_{C_h}$ has state space
$\calS_{C_h}$ and an arbitrary initial distribution
$\lambda_{C_h}=(\lambda_{C_h}(s_{C_h}): s_{C_h} \in \calS_{C_h})$, for each $h=1, \ldots,r$.

Moreover, by replacing $\calS$ with $\calS_{C_h}$ and
$\mu$ with $\mu_{C_h}$, we replicate the construction  provided
before Theorem \ref{dinamic2}. In so doing, we take $k \in \N$
to define the distribution $\mu_{C_h,k}=(\mu_{C_h,k}(s_{C_h}):
s_{C_h} \in \calS_{C_h})$.
% such that $ || \mu_{C_h,k} - \mu_{C_h} ||_{TV} \leq \frac{1}{k}$.

Now, take a sequence of increasing times $(t_\ell^{(C_h)}:
\ell \in \N)$, such that 
\begin{equation}\label{minimo}
\min_{h = 1, \ldots ,
r}t^{(C_h)}_{\ell+1}-t^{(C_h)}_\ell \geq c \ell^{5N-1}, 
\end{equation}
with $c $ a positive constant.

 The transition matrices of $X_{C_h}$ are $(P_{C_h}(t):t \in \N)$
as in \eqref{Penne}:
\begin{equation}\label{Penne2}
    P_{C_h}(t) = \sum_{k =1}^{\infty} P^{(\mu_{C_h,k},G_h)} \mathbf{1}_{\{ t \in [t_k^{(C_h)}
, t_{k+1}^{(C_h)} ) \}}.
\end{equation}

We introduce the Markov chain $X=(X(t) =(X_{C_1}(t), \dots, X_{C_r} (t))\in \calS : t\in \N ) $.
% with state space $\calS$. %$X$ will be used in the next result. 
%The value  $X (t)$ is the Cartesian product of $ X_{C_1} (t) , \ldots ,  X_{C_r}  (t) $ 

Next result is similar to Theorem \ref{dinamic2} but it
is based on the $r$ independent Markov chains constructed above. In the context of MCMC, this framework provides a
remarkable simplification of the computational complexity, 
in that dealing with $r$ independent
Markov chains with state spaces $\calS_{C_1}, \ldots, \calS_{C_r}$ is more affordable than
only one Markov chain with state space given by
$\calS=\calS_{C_1}\times \ldots \times \calS_{C_r}$. Furthermore, as
we will see and as preannounced above, such a context will be of theoretical usefulness in
defining the repeated $G$-games.

\begin{theorem}\label{dinamic3}
Let us consider a finite set $V $ and a  partition $\calC =\{ C_1, \ldots , C_r\}$ of $V$.    Let $\calS = \prod_{\ell =1}^r  \calS_{C_\ell}$ and  a
$\calC$-decomposable connected graph $G=(\calS,E) $, with
$\calC$-decomposition $\calG=(G_1, \ldots, G_r)$.   % , with$G_h=(\calS_{C_h},E_h)$, for $h=1, \dots,r$.  

Let us take a %$\calC$-mixed equilibrium
product distribution $\mu=\prod_{h=1}^r \mu_{C_h}$ and consider the
$r$ independent Markov chains $X_{C_1}=(X_{C_1}(t):t \in \N), \dots,
X_{C_r}=(X_{C_r}(t):t \in \N)$ constructed above, and the Markov
chain $X=    ( (X_{C_1} (t), \dots, X_{C_r}  (t) : t \in \N )$.

Then
\begin{equation}\label{XX11}
    \lim_{t \to \infty}\frac{1}{t}\sum_{m=0}^{t-1}
    \mathbf{1}_{\{X(m)=s\}}=\lim_{t \to \infty}\frac{1}{t}\sum_{m=0}^{t-1} \prod_{h=1}^r
    \mathbf{1}_{\{X_{C_h}(m)=s_{C_h}\}}=\prod_{h=1}^r
    \mu_{C_h}(s_{C_h})=\mu(s),
\end{equation}
for each $s=(s_{C_1}, \dots, s_{C_r}) \in \calS$.
\end{theorem}
\begin{proof}
By \eqref{minimo}  follows that
$$
\lim_{t \to \infty} \frac{   \left   | [0,t] \cap \left (
\bigcup_{h=1}^r \bigcup_{\ell=1}^\infty [   t^{(C_h)}_\ell ,
t^{(C_h)}_\ell + \ell^{2N}) \right )   \right       |}{t}   =0.
$$
In fact, for each $h = 1, \ldots , r$,
$$
\lim_{t \to \infty} \frac{\left | [0,t] \cap \left (
\bigcup_{\ell=1}^\infty [   t^{(C_h)}_\ell , t^{(C_h)}_\ell +
\ell^{2N}) \right ) \right |}{t}   =0,
$$
since
$$
\lim_{\ell \to \infty}\frac{\ell^{2N}}{t^{(C_h)}_{\ell + 1}
-t^{(C_h)}_{\ell } }\leq \lim_{\ell \to \infty}\frac{\ell^{2N}}{c
\ell^{5N -1} } =0.
$$
Thus, the times in $\bigcup_{h=1}^r \bigcup_{\ell=1}^\infty [
t^{(C_h)}_\ell , t^{(C_h)}_\ell + \ell^{2N}) $ can be neglected in
the procedure of checking \eqref{XX11}, i.e.
$$
    \lim_{t \to \infty}\frac{1}{t}\sum_{m=0}^{t-1}
    \mathbf{1}_{\{X(m)=s\}}=
    \lim_{t \to \infty}\frac{1}{t}\sum_{m=0}^{t-1}
    \mathbf{1}_{\{X(m)=s\}}\cdot \mathbf{1}_{\{m \notin
\bigcup_{h=1}^r \bigcup_{\ell=1}^\infty [   t^{(C_h)}_\ell ,
t^{(C_h)}_\ell + \ell^{2N})  \}}
$$
and also
$$
    \lim_{t \to \infty}\frac{1}{t}\sum_{m=0}^{t-1}
    \mathbf{1}_{\{X(m)=s\}}=
    \lim_{t \to \infty}\frac{1}{t}\sum_{m=0}^{t-1}\left[
    \mathbf{1}_{\{X(m)=s\}}+ \mathbf{1}_{\{m \in
\bigcup_{h=1}^r \bigcup_{\ell=1}^\infty [   t^{(C_h)}_\ell ,
t^{(C_h)}_\ell + \ell^{2N})  \}}\right].
$$
Let us define the set of times $A =\bigcup_{h =1}^r
\bigcup_{\ell=1}^\infty   [ t^{(C_h)}_\ell , t^{(C_h)}_\ell +
\ell^{2N})$. 
Now we introduce the independent  random variables  
$( Y_{C_h}  ( t  )  :  t \in \N , h= 1, \ldots , r)$.  The random variables $( Y_{C_h}  ( t  )  :  t \in \N )$,  with  label $h $, take 
value on $S_{C_h}$.  Moreover,  if 
$    t \in[t^{(C_h)}_k , t^{(C_h)}_{k+1})$ then
   $Y_{C_h}  ( t  )  $  has distribution $\mu_{C_h,k}$.

%Given $h = 1, \ldots, r$, for any time not belonging to $
%\bigcup_{\ell=1}^\infty [ t^{(C_h)}_\ell , t^{(C_h)}_\ell +
%\ell^{2N})$, one can define the auxiliary variables $Y_{C_{1}},
%\ldots, Y_{C_r}$, where
%$Y_{C_h}= (Y_{C_h}(t):t \in \N)$ has values on $\calS_{C_h}$ such
%that $Y_{C_h}(i)$ has distribution $\mu_{C_h,k}$ when $i \in
%[t^{(C_h)}_k , t^{(C_h)}_{k+1})$. Assume that the collections of %random variables $Y$'s are independent.

We now adapt formula \eqref{X=Y} to the Markov chain $X_{C_h}$.
If $\bar{t} \notin A$ then for each $h = 1, \ldots , r $ there exists  $ \bar{\ell} _h $ such that $\bar{t} $ belong to $  [ t^{(C_h)}_{\bar{\ell}_h} , t^{(C_h)}_{\bar{\ell}_h+1}  )$. 
In this case formula \eqref{X=Y} becomes 
\begin{equation}\label{Acapp22}
    \mathbb{P} ( X_{C_h}(\bar t) = Y_{C_h}(\bar t) )
 \geq 1 -  \exp \left (-\hat{c}_N \left
\lfloor\frac{{\bar \ell_h}^{N+1}}{N-1} \right \rfloor \right ),
\end{equation}
where we recall that $\hat{c}_N=\frac{1}{[2(N-1)^2]^{N-1}}$.

Hence, for any $\bar t \notin A$ one has that there
exist $\bar \ell_1,\ldots, \bar \ell_r \in \N$ such that $\bar t \in
\bigcap_{h=1}^r[ t^{(C_h)}_{\bar\ell_h} + {\bar \ell_h}^{2N} ,
t^{(C_h)}_{\bar \ell_h +1} ) $. Therefore, using the independence of the random variables $Y$'s and the independence 
of the Markov chains $X$'s, one has  
\begin{equation}\label{Acapp23}
    \mathbb{P} ( (X_{C_1}(\bar t), \ldots,
    X_{C_r}(\bar t))= (Y_{C_1}(\bar t), \ldots,  Y_{C_r}(\bar t)) )
 \geq  1 - \sum_{h=1}^r \exp \left (-\hat{c}_N \left
\lfloor\frac{{\bar \ell_h}^{N+1}}{N-1} \right \rfloor \right ).
\end{equation}
For $\bar t \in
\bigcap_{h=1}^r[ t^{(C_h)}_{\bar\ell_h} + {\bar \ell_h}^{2N} ,
t^{(C_h)}_{\bar \ell_h +1} ) 
$, the distribution
of $(Y_{C_1}(\bar t), \ldots, Y_{C_r}(\bar t))$ coincides 
with $\prod_{h=1}^r\mu_{C_h,\bar {\ell}_h}$.

Thus, we have
\begin{equation}\label{TV2}
\left|\left|\mu-\prod_{h=1}^r\mu_{C_h,\bar \ell _h}\right|\right|_{TV}\leq
\sum_{h=1}^r\frac{1}{\bar \ell_h}.
\end{equation}
Notice that any $\bar \ell _h $ increases to infinity when $\bar t$ goes to infinity. Therefore, the left-hand side of \eqref{TV2} goes to zero
as $\bar t$ goes to infinity.
Inequalities \eqref{Acapp23} and \eqref{TV2} give 
an upper bound for the distance in total variation between the law of $X (\bar t )$ and the distribution $\mu$.

Now, by following the arguments in the proof of Theorem \ref{dinamic2}, % and applying the first Borel-Cantelli lemma 
we obtain equation \eqref{XX11}.
\end{proof}

 In the same setting of the previous theorem,
 consider a graph $G = (\calS, E ) $ and a 
 connected graph $G' =(\calS, E')$ such that 
 $E' \subset E$ and $G' $ is
$\calC $-decomposable. Theorem \ref{dinamic3}  can 
be applied to the graph $G'$ and the Markov chain $X$.  
In any case $X $ 
is also consistent with the graph $G $, which is connected by construction having more edges than $G'$. 
Thus, in some sense,  Theorem \ref{dinamic3} can be applied also to the supergraph $G$ of $G'$.

\subsection{Repeated $G$-games}

\begin{comment}
In this section we analyse different aspects of the class of repeated games. [XXXX RIVEDERE. METTERE IN INTRO? Such class contains games of different types.
First of all, the definition of repeated game must include also a precise statement on how many times the game is played. In this respect, we notice that a repeated game can be played a finite or infinite number of times. Second, it is important to define the information set (on the past) available at each step of the game for each player. Third, one needs to assess the presence of "path-dependence" and whenever the starting point of the game is known and how it is selected. Fourth, the (time-dependent) rule for passing from a strategy to another one is a further ingredient of a repeated game. XXXX]

\end{comment}

We now give the general definition of a repeated game in presence of a
coalition structure $\calC=\{C_1, \ldots,C_r\}$, %which does not change with time, 
and then we present our specific setting.

\begin{definition}\label{ripetuti}
A \emph{repeated game} is a game played $T$ times, with $T \in \N \cup
\{\infty \}$, by a set of players $V$ with a
coalition structure $\calC=\{C_1, \ldots,C_r\}$, where 
the coalition $C
\in \calC$ has set of strategies $\calS_{C}$. 
Each coalition $C\in \calC$ has a initial strategy $s_C(0)$ at time $t=0$. %, which can be assigned or constructed under a %specific criterion. 
Furthermore, coalition $C$ selects at time $1\leq t < T$ a strategy
$s_{C}(t) \in \calS_C$, where such a selection can depend only on the available information on the
previous history of the game. The payoff function of $C \in \calC$
at any time is
given by $\Pi_{C}:\calS \to \R$. %We denote the payoff of the
%coalition $C$ at time $t$ as $\Pi_{C}(t)$. 
The \emph{payoff}
of the $T$ times repeated game for the coalition $C \in \calC$ is
$$
\Pi_C^{(T)}=\left\{%
\begin{array}{ll}
    \frac{1}{T}\sum_{m=0}^{T-1}\Pi_C(s_{C_1}(m), \ldots ,s_{C_r}(m) ), & \hbox{if $T<\infty$;} \\
\mbox{} \\
   \liminf_{t \to \infty}\frac{1}{t}\sum_{m=0}^{t-1}\Pi_C (    s_{C_1}(m), \ldots ,s_{C_r}(m)    ),  & \hbox{if $T=\infty$.} \\
\end{array}
\right.
$$
\end{definition}

\medskip  

In the proposed context of $G$-games, each coalition $C$ can select 
at time $t+1$ only a strategy (or action)
 $s_C(t+1)$ that is adjacent to the
strategy $s_C(t)$ selected at time $t$. Such a
requirement cannot be satisfied for a general graph, because it
would imply the construction of a 
strategy $s_C(t+1)$ by knowing the decisions that are doing the other coalitions.  
 To convince the reader of this problem, we present a simple
example in the setting of two players.
\begin{example}
\label{ex:rep} Consider $V=\{1,2\}$, $\calS=\calS_1 \times \calS_2$,
with $ \calS_1 \equiv \calS_2 \equiv \{s_1,s_2\}$. The graph is
$G=(\calS, E)$, where $$E=\{((s_1,s_1),(s_2,s_1)),
\,((s_1,s_1),(s_1,s_2)),\, ((s_1,s_2),(s_2,s_2)),\,
((s_2,s_1),(s_2,s_2))\}.$$
Notice that $(s_1,s_1)$ and $(s_2,s_2)$ are not
adjacent, and therefore it is impossible to have that $s(t)=(s_1,s_1)$ and $s(t+1)=(s_2,s_2)$. In any case, both players have the opportunity to move from $s_1$
 to $s_2$ if the other player decides to remain in $s_1$.

Thus, the selection of the strategy by a player 
should not depend only
on the past, but also on the current choices  of the other player. However, this situation is not considered in Definition \ref{ripetuti}. %therefore we need to restrict our attention to a particular class of graphs. 
\end{example}

For the reasons expressed above and explained through Example
\ref{ex:rep} we will define the repeated $G$-games when $G$  can be written as product 
of graphs according to Definition~\ref{decomponibile}.

We are ready to present the definition of repeated $G$-games.
%provided here in the general case of presence of a generic %coalition structure $\calC$.

\begin{definition}
\label{def:repeated} Consider a connected graph $G=(\calS,E)$, a
set of players $V$ and a coalition structure  $\calC
=\{ C_1, \ldots , C_r\}$. Assume that $\calG=(G_1, \ldots, G_r)$ is
the $\calC$-decomposition of $G$. A \emph{repeated $G$-game} 
is a repeated game
such that, for $t \in \N$, any coalition $C \in \calC$  can select at time $t+1$ only strategies $s_{C}(t+1)$'s in
$\calS_C$ which are adjacent to the strategy
$s_{C}(t) \in \calS_C$ selected at time $t$.
%We denote such a game as $(V,\Pi;G,\calC)$ and
%will call it simply as repeated $G$-game.
\end{definition}
\medskip 

Let us focus on the information. We will present two extreme cases. 
\begin{itemize}
\item[$(I_M)$] All the coalitions are aware about the previous history of the repeated game. This is the maximum available level of information, called \emph{maximal information}.  
\item[$(I_m)$] Any coalition has knowledge of time $t$ and of its previously selected strategies before $t$, without any information on the choices of the others coalitions. 
We call this case \emph{minimal information}.  
\end{itemize}

We also assume that the vector of the initial strategies of the coalitions at time zero are of two types.
\begin{itemize}
\item[$(P_0)$] The initial strategies are decided by the players themselves.
\item[$(R_0)$] There is a referee of the game who assigns the initial strategies.
\end{itemize} 
In both cases when the information is minimal the coalitions are not aware 
about the initial strategies of the others.  

We notice that it is often not possible or really hard to load the entire past history of the game into memory. Therefore, we will  
focus mainly on Markovian processes, where the only datum needed is the current  time $t $ and the  strategy selected at time $t-1$.

%
%The concept of $\calC$-decomposable graphs is relevant because, as
%we will see, it allows us to deal with $\calC$-mixed equilibria in a
%context of repeated $G$-game, in presence of coalitions which do not
%interact. XXXX COLLEGARE CON LE SOLUZIONI DECENTRALIZZATE? XXXX
%

%
%The only interesting cases are those where the graph $G$ is
%connected. Therefore, in this Section we assume that $G $ is
%connected.

%As we will see in Theorem \ref{ultimo}, the distribution $\mu $ in
%Theorem \ref{dinamic3} is typically the $\calC$-mixed equilibrium
%for a $G$-game whose ingredients are described above. 

\medskip

We now present an immediate 
result showing the relevance of the pure $\calC$-equilibria, given  in Definition \ref{C-equilibrium},  for $T=2$. 

\begin{theorem}
\label{guess}
Consider a repeated $G$-game with coalition 
structure $\calC=\{ C_1, \ldots  , C_r \}$ where $T=2$, information is of 
type $(I_M)$ and initial strategies are of type $(R_0)$. If 
the vector of initial strategies $\bar s =(\bar s_{C_1} , \ldots ,\bar s_{C_r} ) \in \calE_\calC$ , then 
it is a 
Berge equilibrium for $\calC$ at time $t=1$ where the space of strategies available to any coalition 
$C \in \calC$ is restricted to the strategies adjacent to $\bar s_{C} $. 
\end{theorem}

 We also believe that $\calC$-equilibria, defined in Section \ref{Sec2},  
 become relevant  when the coalitions are not aware 
 about the terminal time of the game. This is the case 
 of a final stage of the game not under the control of the 
 single coalitions. For instance, each 
 coalition has the opportunity to abandon the game and 
 such an abandonment would determine the end of the 
 game itself. In a different context, the game might 
 end at the occurrence of an event whose distribution 
 is not known. In all such situations one can reasonably 
 guess that coalitions do not consider the consequences 
 on the long-term of their strategies and 
they play the game by implicitly assuming that each round is the last one. 
 % (in a different setting, such a remark can be find also in \cite{KMR93}).
 Thus, it is reasonable to believe that if the coalitions achieve at time $\bar t$ a $\calC$-equilibrium, then they will play such strategies at each time $t >\bar t$. Indeed, such a way to play leads to a stage-wise maximization of their payoffs.

%\end{rem}
\begin{comment}

\textbf{The achieved $\calC$-equilibrium is then comparable to the
\emph{evolutionary stable equilibrium} of \cite{animal}, where once
an equilibrium strategy is achieved, then players do not have
benefits from selecting a different strategy in the subsequent steps
of the game.}

XXXXXXXXXXXXX
Vedere Learning, Mutation, and Long Run Equilibria in Games
BY MICHIHIRO KANDORI, GEORGE J. MAILATH, AND RAFAEL ROB  
    \cite{KMR93}.

 "The myopia assumption amounts to saying that at
the same time that players are learning, they are not taking into account the
long run implications of their strategy choices. 
Thus, agents act as if each stage
game is the last. 
Mutation plays a central role in our analysis. 
With some small probability,
each agent plays an arbitrary strategy.10 One 
economic interpretation is that a
player exits with some probability and is replaced with 
a new player who knows
nothing about the game and so chooses a strategy at 
random (as in Canning (1989))." 

Evolutionary games on graphs
Gyorgy Szabo a Gabor Fath  \cite{szabo2007}

 Maynard Smith and Price (1973) have introduced the concept of ESS (evolutionarily
stable strategy) 
XXXXXXXXXXXXX
\end{comment}

\medskip

We are ready to give the definition of  equilibrium for the repeated $G$-games. According to the specific framework we will focus on, we restrict our attention to the case of minimal information and $T = \infty$.

% At this aim, we need to point out that the generic stochastic processes we 
%considered are adapted to the available information on the past. 
\begin{definition}
\label{def:CGR}
Consider a connected graph $G=(\calS,E) $ and a repeated $G$-game with coalition structure $\calC=\{C_1, \dots, C_r\}$.
Assume that $\calG=(G_1, \ldots, G_r)$ is
the $\calC$-decomposition of $G$. Assume $T=\infty$, minimal information and initial strategies of type $(P_0)$ or $(R_0)$.  Consider $r$ adapted and independent stochastic processes $X_{C_1}=(X_{C_1}(t):t \in \N), \dots,
X_{C_r}=(X_{C_r}(t):t \in \N)$ taking values in $\calS_{C_1}, \dots,\calS_{C_r}$ and consistent with $G_1, \ldots, G_r$, respectively.

We say that $X=(X_{C_1},\ldots, X_{C_r})$ is a $\calC$-equilibrium for the repeated $G$-game when
$$
\mathbb{E}(\liminf_{t \to \infty}\frac{1}{t} \sum_{j=0}^{t-1}
\Pi_{C_\ell}(X_{C_1}(m),\ldots,X_{C_{r}}(m))) \geq
$$
\begin{equation}\label{eqGgame}  \geq \mathbb{E}(\liminf_{t \to \infty}\frac{1}{t} \sum_{j=0}^{t-1}
\Pi_{C_\ell}(X_{C_1}(m),\ldots,X_{C_{\ell-1}}(m),\tilde{X}_{C_{\ell}}(m),X_{C_{\ell+1}}(m), \ldots,X_{C_{r}}(m))),
\end{equation}
for each $\ell=1, \dots, r$ and for any adapted stochastic process $\tilde{X}_{C_{\ell}}$ which is consistent with $G_\ell$ and independent from $ X_{C_1},\ldots,X_{C_{\ell-1}}, X_{C_{\ell+1}}, \ldots,X_{C_{r}} $.
\end{definition} 

The previous definition is meaningful only in the case of minimal information because we are requiring 
the independence of the strategy processes.
We now illustrate the connection between the equilibria in $\calM_\calC$ and the $\calC$-equilibria for the repeated $G$-games defined in Definition \ref{def:CGR}. 
Specifically, we will present a version of the Folk Theorem for our framework.
To proceed, we need a preliminary general result. We state it directly in the language of the $G$-games, for the sake of notation. 
\begin{theorem}
\label{thm:calM}
Consider a connected graph $G=(\calS,E) $ and a repeated $G$-game with coalition structure $\calC=\{C_1, \dots, C_r\}$. Assume $T=\infty$ and that $\calG=(G_1, \ldots, G_r)$ is
the $\calC$-decomposition of $G$.
% and initial strategies of type $(P_0)$ or $(R_0)$. 
Consider $\Lambda_{C_1} \times \ldots \times \Lambda_{C_r} \in \calM_\calC$ and $r$ independent Markov chains $X_{C_1}=(X_{C_1}(t):t \in \N), \dots,
X_{C_r}=(X_{C_r}(t):t \in \N)$ with state space $\calS_{C_1}, \dots,\calS_{C_r}$ and consistent with $G_1, \ldots, G_r$, respectively, such that
\begin{equation}\label{calM-1}
 \lim_{t \to \infty}\frac{1}{t}\sum_{m=0}^{t-1} \prod_{h=1}^r
    \mathbf{1}_{\{X_{C_h}(m)=s_{C_h}\}}=\prod_{h=1}^r
    \Lambda_{C_h}(s_{C_h}),  \qquad a.s.,
\end{equation}
for each  $s_{C_h} \in \calS_{C_h}$.

Then, almost surely,
$$
\Pi_{{C_\ell}}^{(\infty)} = \mathbb{E}_{\Lambda_{C_1}  \times \ldots  \Lambda_{C_r}   }  (\Pi_{C_\ell})        =\lim_{t \to \infty} \frac{1}{t} \sum_{m=0}^{t-1}\Pi_{C_\ell}(X_{C_1}(m),\ldots,X_{C_r}(m)) \geq 
$$
\begin{equation}\label{eqult1} \geq    \mathbb{E}(\liminf_{t \to \infty}\frac{1}{t} \sum_{j=0}^{t-1}
\Pi_{C_\ell}(X_{C_1}(m),\ldots,X_{C_{\ell-1}}(m),\tilde{X}_{C_{\ell}}(m),X_{C_{\ell+1}}(m), \ldots,X_{C_{r}}(m))).
\end{equation}
for each $\ell=1, \dots, r$ and stochastic process $\tilde{X}_{C_{\ell}}$ consistent with $G_\ell$ and independent from  the  Markov chains $ X_{C_1},\ldots,X_{C_{\ell-1}}, X_{C_{\ell+1}}, \ldots,X_{C_{r}} $. 
\end{theorem}

\begin{proof} 
By formula \eqref{ergo1} we deduce the first two equalities in \eqref{eqult1}, so we 
have to prove only the inequality in \eqref{ergo1}. 

 Let us take  $\ell = 1, \ldots, r$. 
Consider a stochastic process $\tilde{X}_{C_{\ell}}$ 
consistent with $G_\ell$ and independent from  the  Markov chains $ X_{C_1},\ldots,X_{C_{\ell-1}}, X_{C_{\ell+1}}, \ldots,X_{C_{r}} $. 
 From the finiteness of $\calS$ one has the tightness of the distributions on $\calS$. Therefore,  there exists a sequence $(t_n:n \in \N)$ and a distribution $\tilde{\Lambda}_{C_\ell}$ on $\calS_{C_\ell}$ such that
\begin{equation}\label{calM-2}
    \lim_{n \to \infty}\frac{1}{t_n}\sum_{m=0}^{t_n-1}
    \mathbf{1}_{\{\tilde{X}_{C_\ell}(m)=s_{C_\ell}\}}=\tilde{\Lambda}_{C_\ell}(s_{C_\ell}), \qquad a.s.,
\end{equation}
for each $s_{C_\ell} \in \calS_{C_\ell}$.

The independence assumption of $\tilde{X}_{C_\ell}$ from $X_{C_1},\ldots,X_{C_{\ell-1}}, X_{C_{\ell+1}}, \ldots,X_{C_{r}}$ and \eqref{calM-2} give 
$$
 \mathbb{E}  ( \lim_{n \to \infty}\frac{1}{t_n}\sum_{m=0}^{t_n-1}
    \mathbf{1}_{\{(X_{C_1}(m),\ldots,X_{C_{\ell-1}}(m),\tilde{X}_{C_\ell}(m), X_{C_{\ell+1}}(m), \ldots,X_{C_{r}}(m)) 
=s\}}   )=
$$
\begin{equation}\label{calM-3}
=\tilde{\Lambda}_{C_\ell}(s_{C_\ell}) \left[ \prod_{k=1, \dots, r; \,k \not= \ell}{\Lambda}_{C_k}(s_{C_k}) \right], 
\end{equation}
for any  $s=(s_{C_1}, \dots, s_{C_r})   \in \calS$.
Thus,  \eqref{calM-3} leads to 
$$
 \mathbb{E}  ( \lim_{n \to \infty}\frac{1}{t_n}\sum_{m=0}^{t_n-1}\Pi_{C_\ell}(X_{C_1}(m),\ldots,X_{C_{\ell-1}}(m),\tilde{X}_{C_{\ell}}(m),X_{C_{\ell+1}}(m), \ldots,X_{C_{r}}(m))   )=
$$
\begin{equation}\label{calM-4}
=\mathbb{E}_{\Lambda_{C_1}  \times \ldots  \Lambda_{C_{\ell-1}} \times  \tilde{\Lambda}_{C_{\ell}} \times  \Lambda_{C_{\ell+1}} \times \ldots \times  \Lambda_{C_{r}}   }  (\Pi_{C_\ell})  . 
\end{equation}
By the assumption that $\Lambda_{C_1}  \times \ldots   \times  \Lambda_{C_{r}}  \in \calM_\calC $, one has that 
$$
\mathbb{E}_{\Lambda_{C_1}  \times \ldots  \Lambda_{C_{\ell-1}} \times  \tilde{\Lambda}_{C_{\ell}} \times  \Lambda_{C_{\ell+1}} \times \ldots \times  \Lambda_{C_{r}}   }  (\Pi_{C_\ell})  \leq \mathbb{E}_{\Lambda_{C_1}  \times \ldots   \times  \Lambda_{C_{r}}   }  (\Pi_{C_\ell}) .
$$
The thesis comes from
$$
    \liminf_{t \to \infty}\frac{1}{t}\sum_{m=0}^{t-1}\Pi_{C_\ell}(X_{C_1}(m),\ldots,X_{C_{\ell-1}}(m),\tilde{X}_{C_{\ell}}(m),X_{C_{\ell+1}}(m), \ldots,X_{C_{r}}(m)) \leq 
$$
$$
\leq
 \lim_{n \to \infty}\frac{1}{t_n}\sum_{m=0}^{t_n-1}\Pi_{C_\ell}(X_{C_1}(m),\ldots,X_{C_{\ell-1}}(m),\tilde{X}_{C_{\ell}}(m),X_{C_{\ell+1}}(m), \ldots,X_{C_{r}}(m)).
$$
\end{proof}

The previous theorem  says that the 
considered Markov chains form a $\calC$-equilibrium 
for the repeated  $G$-game when $T=\infty$, initial 
strategies of types $(P_0)$ or $(R_0)$ and minimal information (see Definition \ref{def:CGR}). %Specifically, such an equilibrium is the probability distribution appearing in the right-hand side of \eqref{calM-1}
From Theorem \ref{dinamic3} we know that such Markov chains exist and we have constructed them. We point out that Equation \eqref{calM-1} represents a global condition on the empirical distribution related to all the coalitions. 
Differently, Theorem \ref{dinamic3} contains a local condition which actually implies \eqref{calM-1}. In fact,
in Theorem \ref{dinamic3} we have constructed $r$ independent Markov chains such that if
$$
 \lim_{t \to \infty}\frac{1}{t}\sum_{m=0}^{t-1}
    \mathbf{1}_{\{X_{C_\ell}(m)=s_{C_\ell}\}}=
    \Lambda_{C_\ell}(s_{C_\ell}),  \qquad a.s.,
$$ 
for each $\ell=1, \ldots, r $, then one obtains also  \eqref{calM-1}. This is relevant in game theory at least in the case of minimal information. 
%We also notice that Theorem \ref{thm:calM} holds in any case of information set. 
In fact, the independence assumption of the considered Markov chains is always valid  under the condition of minimal information. In fact, in such cases, any coalition does not know the strategies selected at the previous steps by the other coalitions playing the game.

By Theorems \ref{dinamic3} and Theorem \ref{thm:calM}, we obtain the following.

\begin{corollary}
\label{coro1}
Consider a connected graph $G=(\calS,E) $ and a repeated $G$-game with coalition structure $\calC=\{C_1, \dots, C_r\}$. Assume $T=\infty$, initial strategies of types $(P_0)$ or $(R_0)$ and information of type $(I_m)$. Assume also that $\calG=(G_1, \ldots, G_r)$ is
the $\calC$-decomposition of $G$.

For any $\Lambda_{C_1} \times \ldots \times \Lambda_{C_r} \in \calM_\calC$, there exist $r$ independent Markov chains $X_{C_1}=(X_{C_1}(t):t \in \N), \dots,
X_{C_r}=(X_{C_r}(t):t \in \N)$ with state space $\calS_{C_1}, \dots,\calS_{C_r}$ and consistent with $G_1, \ldots, G_r$, respectively, such that $X=(X_{C_1}, \dots, X_{C_r})$ is a $\calC$-equilibrium 
for the repeated $G$-game. 
\end{corollary}

Corollary \ref{coro1} is a version of the Folk Theorem 
in our framework, which guarantees the existence 
of $\calC$-equilibria for the repeated game. It is also 
important that the $\calC$-equilibria analysed are 
Markovian that means that the single coalition has to memorize only the current time and the 
strategy played in the previous stage of the game. 

As a by-product, Corollary \ref{coro1} can be also related to the memory $K$ of the single coalitions -- i.e., the knowledge by a coalition $C\in \calC$ of the strategies played at the previous $K$ stages of the repeated $G$-games by $C$. In particular, it states  that the $\calC$-equilibria are equilibria for any repeated $G$-games where any coalition has at least memory $K=1$.

\section{Conclusions}
This paper has introduced and analyzed the $G$-games, i.e. games
whose strategies are nodes of a graph $G$. This class of games
presents interesting features either under a theoretical as well as
under a practical point of view.

Indeed, several real-world situations can be modeled through game
models where the players are physically constrained to move
sequentially from a strategy to an adjacent one. The introduction of
a graph of the strategies serves to capture such constraints.
We specifically deal with $G$-games with a coalition structure. In
so doing, we admit the presence of interactions among the players.
Notice that the presented framework does not
exclude the case of absence of constraints -- one can take $G$
complete -- and the possibility of noncooperative games -- by taking
coalitions formed by single players.

A detailed exploration of several aspects is carried out. In
particular, we define the equilibria of the coalitions of pure and
mixed type and present a version of the Folk Theorem for the
class of infinitely played $G$-repeated games with connected $G$.
The definition of $\calC$-equilibria represents an extension of the
(pure) Berge and Nash equilibria. However, we do not enter here the
challenging problem of equilibria selection (see the seminal work \cite{harselten} and e.g. \cite{binmore, car93, duffy, harsanyi, KMR93, sch80}), leaving this topic for future research.

It is important to note that Theorem \ref{dinamic3} can be seen as
an \emph{universality result}. In fact, assume that multiple
equilibria are attained and a selection criterion identifies one of
them as the valid one for all the coalitions. Then, such a valid
equilibrium can be achieved under the requirement that $G$ is connected.

We also point out that coalitions are fixed in the developed game
model. The theme of coalition structure generation -- i.e., the
problem of partitioning the players, according to a specific
criterion (see e.g. \cite{sandhlom1999, sandhlom}) -- is beyond the
material presented in this paper. Also such a challenging research
theme is left for future development of the study of the $G$-games.

\bibliographystyle{abbrv}

%\bibliography{EmilioRoy}

\end{document}